\documentclass{amsart}

\usepackage{amssymb,amsmath,graphicx,csquotes}

\newtoks\prt

\numberwithin{equation}{section}

\newtheorem{thm}{Theorem}[section]

\newtheorem{lemma}[thm]{Lemma}

\newtheorem{cor}[thm]{Corollary}

\theoremstyle{definition}

\newtheorem{remark}[thm]{Remark}

\def\eqn#1$$#2$${\begin{equation}\label#1#2\end{equation}}

\def\fra{\mathfrak{A}}

\def\C{\mathcal C}

\def\F{\mathcal F}

\def\M{\mathcal M}

\def\ce{\mathbb C}

\def\lin{Lindel\"of}

\def\co{\operatorname{co}}

\def\ep{\varepsilon}

\def\en{\mathbb N}
\def\ef{\mathbb F}

\def\er{\mathbb R}

\def\r{|}
\def \aff{\operatorname{Aff}}
\def\ov{\overline}

\def \lin {\operatorname{Lin}}

\def \Bos {\operatorname{Bos}}
\def \Hs {\operatorname{Hs}}
\def \Bas {\operatorname{Bas}}

\def \id {\operatorname{id}}

\def \ext {\operatorname{ext}}

\def \reg {\partial _{\kern1pt\text{reg}}}

\def\di{\,\mbox{\rm d}}

\newcommand{\norm}[1]{\left\|#1\right\|}

\renewcommand{\Re}{\operatorname{Re}}

\renewcommand{\Im}{\operatorname{Im}}

\newcommand{\setsep}{;\,}
\newcommand{\fr}{Fr\'echet\ }

\begin{document}

\title{Preserving affine Baire classes by perfect affine maps}

\author{Ond\v{r}ej F.K. Kalenda and Ji\v r\'\i\ Spurn\'y}

\address{ Charles University in Prague\\
Faculty of Mathematics and Physics\\Department of Mathematical Analysis \\
Sokolovsk\'{a} 83, 186 \ 75\\Praha 8, Czech Republic}
\email{kalenda@karlin.mff.cuni.cz}
\email{spurny@karlin.mff.cuni.cz}

\subjclass[2010]{46A55; 26A21; 54H05}

\keywords{vector-valued Baire function; strongly affine function}

\thanks{Our investigation was supported by the Research grant GA\v{C}R P201/12/0290. The second author was also
supported by The Foundation of Karel Jane\v{c}ek for Science and Research.}

\begin{abstract}
Let $\varphi\colon X\to Y$ be an affine continuous surjection between compact convex sets. Suppose that the canonical copy of the space of real-valued affine continuous functions on $Y$ in the space of real-valued affine continuous functions on $X$ is complemented. We show that if  $F$ is a topological vector space, then $f\colon Y\to F$ is of affine Baire class $\alpha$ whenever the composition $f\circ\varphi$ is of affine Baire class $\alpha$. This abstract result is applied to extend
known results on affine Baire classes of strongly affine Baire mappings.
\end{abstract}

\maketitle

\section{Introduction and the main results}

Let $\varphi\colon X\to Y$ be a continuous surjection between compact Hausdorff spaces. If $f\colon Y\to T$ is a mapping with values in a topological space $T$, some properties of $f$ can be deduced from the properties of $f\circ\varphi$. In particular,
since $\varphi$ is a closed mapping, it is easy to check that $f$ is continuous if and only if $f\circ \varphi$ is continuous.
Analogous statements hold for Borel measurable, Baire measurable and resolvably measurable mappings due to \cite{hosp}. We formulate them in the following theorem where we use the notation from \cite{spurny-amh}.

\smallskip
\noindent{\bfseries Theorem A.} {\itshape  Let $\varphi\colon X\to Y$ be a continuous surjection between compact Hausdorff spaces, $f\colon Y\to T$ be a mapping with values in a topological space $T$ and $\alpha<\omega_1$.
\begin{itemize}
	\item[(1)] $f$ is Borel measurable if (and only if) $f\circ\varphi$ is Borel measurable. Moreover, $f$ is $\Sigma_\alpha(\Bos(Y))$-measurable if (and only if) $f\circ\varphi$ is $\Sigma_\alpha(\Bos(X))$-measurable.
 	\item[(2)] $f$ is resolvably measurable if (and only if) $f\circ\varphi$ is resolvably measurable.
 	Moreover, $f$ is $\Sigma_\alpha(\Hs(Y))$-measurable if (and only if) $f\circ\varphi$ is $\Sigma_\alpha(\Hs(X))$-measurable.
	\item[(3)] $f$ is Baire measurable if (and only if) $f\circ\varphi$ is Baire measurable.
	Moreover, $f$ is $\Sigma_\alpha(\Bas(Y))$-measurable if (and only if) $f\circ\varphi$ is $\Sigma_\alpha(\Bas(X))$-measurable.
\end{itemize}
}\smallskip

Let us explain it in more detail. The notation and terminology follow \cite{spurny-amh}. A mapping $f\colon Y\to T$ is \emph{Borel measurable} if the inverse image of any open set is a Borel set.  Similarly, $f$ is \emph{Baire measurable} if the inverse image of any open set belongs to the Baire $\sigma$-algebra (i.e., to the $\sigma$-algebra generated by cozero sets); and $f$ is \emph{resolvably measurable} if the inverse image of any open set belongs to the $\sigma$-algebra generated by resolvable sets.
Further, $f$ is $\Sigma_\alpha(\Bos(Y))$-measurable if $f^{-1}(U)\in\Sigma_\alpha(\Bos(Y))$ for any open set $U\subset T$. Similarly we define $\Sigma_\alpha(\Bas(Y))$-measurable and $\Sigma_\alpha(\Hs(Y))$-measurable mappings.

The `only if' parts of all the three assertions of Theorem A are obvious.
The `if part' of the assertion (1) follows from \cite[Theorem 10]{hosp} and that of the assertion (2) follows from \cite[Theorem 4]{hosp}. To show the `if part' of the assertion (3) we first observe that by \cite[Theorem 2]{frolik-bulams} the mapping $f$ is Baire measurable whenever $f\circ \varphi$ is Baire measurable and we conclude by using  \cite[Theorem 3.6]{spurny-amh} saying that a mapping is  $\Sigma_\alpha(\Bas(Y))$-measurable if and only if it is Baire measurable and $\Sigma_\alpha(\Bos(Y))$-measurable.

If $T$ is a convex subset of a \fr space, the hierarchy of Baire measurable mappings corresponds to the hierarchy of Baire functions. Since we will need more such hierarchies, we will introduce it in an abstract setting:

Given a set $S$, a topological space $T$ and a family of mappings $\F$ from $S$ to $T$, we define the \emph{Baire classes} of mappings as follows. Let $(\F)_0=\F$. Assuming  that $\alpha\in [1,\omega_1)$ is given and that $(\F)_\beta$ have
been already defined for each $\beta<\alpha$, we set
\begin{multline*}
(\F)_\alpha=\{f\colon S\to T\setsep \text{there exists a sequence } (f_n)\text{ in }\bigcup_{\beta<\alpha} (\F)_\beta
\\ \text{ such that }f_n\to f
\mbox{ pointwise}\}.
\end{multline*}

In particular, if $S$ and $T$ are topological spaces, by $\C_\alpha(S,T)$ we denote the set $(\C(S,T))_\alpha$, where $\C(S,T)$ is the set of all continuous functions from $S$ to $T$. Further, by $\C_\alpha(S)$ we mean $\C_\alpha(S,\er)$.

It is known (see, e.g. \cite[Lemma 3.2]{l1pred}) that, whenever $Y$ is a compact space and $T$ is a convex subset of a \fr space, then a mapping $f:Y\to T$ belongs to $\C_\alpha(Y,T)$ if and only if it is  $\Sigma_{\alpha+1}(\Bas(Y))$-measurable.
Therefore we have the following equivalence:

\smallskip{\noindent \bfseries Theorem B.} {\itshape
Let $\varphi:X\to Y$ be a continuous surjection between compact spaces and $T$ be a convex subset of a \fr space. Let $f\colon Y\to T$ be a mapping and $\alpha<\omega_1$. Then $f\in\C_\alpha(Y,T)$ if (and only if) $f\circ\varphi\in\C_\alpha(X,T)$.}
\smallskip

Now, suppose that $X$ and $Y$ are, moreover, compact convex sets (i.e., compact convex subsets of locally convex spaces) and the continuous surjection $\varphi$ is affine. We address the following question:

\begin{center}\it In this setting, does an analogue of Theorem B hold for affine Baire classes?
\end{center}

Let us first recall the definition of affine Baire classes:
If $Y$ is a compact convex set and $T$ is a convex subset of a topological vector space, by $\fra(Y,T)$ we denote the set of all affine continuous functions defined on $Y$ with values in $T$ and, for $\alpha<\omega_1$ we set $\fra_\alpha(Y,T)=(\fra(Y,T))_\alpha$. Further, $\fra(Y)$ stands for the space $\fra(Y,\er)$ and $\fra_\alpha(Y)$ means $\fra_\alpha(Y,\er)$.

Hence, the precise question we address is the following:

\begin{center}
\it	Assuming that $f\circ\varphi\in\fra_\alpha(X,T)$, is necessarily $f\in\fra_\alpha(Y,T)$?
\end{center}

It is clear that $f$ is affine whenever $f\circ\varphi$ is affine. Therefore the answer is positive in case $\alpha=0$.
It is further positive if $T=\er$ and $\alpha=1$, since in this case the class $\fra_1(Y,\er)$ coincides with the class of all affine functions belonging to $\C_1(Y,\er)$ by a result of Mokobodzki \cite[Theorem 4.24]{lmns}. However, the answer is negative in general. More precisely, for $\alpha\ge2$ it is negative even for scalar functions and for $\alpha=1$ it is negative for vector-valued functions.

Let us explain it more detail. Let $Y$ be an arbitrary compact convex set. If $\mu$ is a Radon probability measure on $Y$,
then there is a unique point $x\in Y$ such that $u(x)=\int u\di\mu$ for any $u\in\fra(Y)$. This unique point is called the \emph{barycenter} of $\mu$ and is denoted by $r(\mu)$ (see \cite[Definition 2.26]{lmns}). Further, a function $f\colon Y\to\er$ is called \emph{strongly affine} if for each Radon probability $\mu$ on $X$ $f$ is  $\mu$-integrable and $\int f\di\mu=f(r(\mu))$. The same definition can be used for $f\colon Y\to F$, where $F$ is a \fr space (see \cite{l1pred}).
Since the vector integral used in the definition is the Pettis one, such $f$ is strongly affine if and only if $\tau\circ f$ is strongly affine for each $\tau\in F^*$ (see \cite[Fact 1.2]{l1pred}). We will use this characterization as a definition.

If we set $X=\M^1(Y)$, the set of all Radon probability measures on $Y$ equipped with the weak$^*$ topology, then $X$ is again a compact convex set and, by \cite[Proposition 2.38]{lmns} the mapping $r\colon\mu\mapsto r(\mu)$ is a continuous affine surjection. Moreover, by \cite[Proposition 6.38]{lmns} the set $X$ is a Bauer simplex (i.e., a Choquet simplex with closed set of extreme points), therefore by \cite[Theorem 2.5]{l1pred} any strongly affine map $g:X\to F$ with values in a \fr space $F$ which belongs to $\C_\alpha(X,F)$ in fact belongs to $\fra_\alpha(X,F)$.

Now, by \cite{talagrand} there is a compact convex set $Y$ and a strongly affine function $f\in\C_2(Y)$ such that $f\notin\bigcup_{\alpha<\omega_1}\fra_\alpha(Y)$. If we choose $X$ and $r$ as in the previous paragraph, then $f\circ r\in\fra_2(X)$. Similarly, by \cite[Theorem 2.2]{l1pred} there is a compact convex set $Y$, a Banach space $F$ and a strongly affine $f\in\C_1(Y,F)$ such that $f\notin\bigcup_{\alpha<\omega_1}\fra_\alpha(Y,F)$. If we choose $X$ and $r$ as above, we get $f\circ r\in\fra_1(X,F)$.

The main result of the present paper is a sufficient condition for a positive answer.	To formulate it we need the following notation. Let $\varphi\colon X\to Y$ be a continuous affine surjection between compact convex sets. We define a mapping $\varphi^*\colon\fra(Y)\to\fra(X)$ by $\varphi^*(f)=f\circ \varphi$. Then $\varphi^*$ is an isometric embedding of $\fra(Y)$ into $\fra(X)$.

\begin{thm}\label{T:main} Let $X$ and $Y$ be compact convex sets, $\varphi\colon X\to Y$ a continuous affine surjection such that $\varphi^*(\fra(Y))$ is a complemented subspace of $\fra(X)$. Let $F$ be a topological vector space. Let $f:Y\to F$ be such that $f\circ \varphi\in\fra_\alpha(X,F)$ for some $\alpha<\omega_1$. Then $f\in\fra_\alpha(X,F)$.
\end{thm}

We point out that this result is quite abstract and that the range space $F$ is just a topological vector space -- no local convexity, metrizability or completeness is required. The proof is given in the next section and is rather elementary.
Of course, the most interesting case is that of \fr range. In this setting we get the following corollary.

\begin{thm}\label{T:appl}
Let either
\begin{itemize}
	\item[(a)] $X=(B_{E^*},w^*)$, where $E$ is a (real or complex) Banach space which is isomorphic to a complemented subspace of a (real or complex) $L_1$-predual $E_1$, or
	\item[(b)] $X$ is a compact convex set such that $\fra(X,\ef)$ is isomorphic to a complemented subspace of an $L_1$-predual $E_1$ over $\ef$.
\end{itemize}
 If $F$ is a \fr space, then any strongly affine $f\in\C_\alpha(X,F)$ belongs to $\fra_{1+\alpha}(X,F)$.
If $\ext B_{E_1^*}$ is moreover $F_\sigma$ in the weak$^*$ topology, $1+\alpha$ can be replaced by $\alpha$.
\end{thm}

Recall that a (real or complex) $L_1$-predual is a (real or complex) Banach space whose dual is isometric to a space of the form $L_1(\mu)$ for some non-negative measure $\mu$. The letter $\ef$ stands for $\er$ or $\ce$.

As an immediate consequence we get the following result.

\begin{cor} \
\begin{itemize}
	\item The Banach space $E$ constructed in \cite{talagrand} is not isomorphic to a complemented subspace of an $L_1$-predual.
	\item Let $X$ be any of the simplices constructed in \cite[Theorem 1.1]{spu-zel}, then $\fra(X)$ is not isomorphic to a complemented subspace of a $\C(K)$ space. (In fact, $\fra(X)$ is even not isomorphic to a complemented subspace of $\fra(Y)$, where $Y$ is a simplex with $\ext Y$ being $F_\sigma$.)
\end{itemize}
\end{cor}

\begin{proof}
By \cite{talagrand} there is a strongly affine function on $(B_{E^*},w^*)$ which is in the class $\C_2$ but not in $\fra_\alpha$ for any $\alpha<\omega_1$. If $X$ is any of the simplices from \cite[Theorem 1.1]{spu-zel}, then there is a strongly affine function $f\in\C_2(X)\setminus\fra_2(X)$. Hence both cases indeed follow from Theorem~\ref{T:appl}.
\end{proof}

\begin{remark}
The complementability condition in Theorem~\ref{T:main} is sufficient but not necessary. For example, if $K$ is any compact space and $F$ is a \fr space, then the space $\fra_\alpha(\M^1(K),F)$ coincide with the subspace of  $\C_\alpha(\M^1(K),F)$ consisting of strongly affine mappings (by \cite[Theorem 2.5]{l1pred}). Therefore, if both $X$ and $Y$ are of the form $\M^1(K)$ for some compact space $K$, the conclusion follows from Theorem B and \cite[Proposition 5.29]{lmns}.
Further, one can easily find $X$ and $Y$ of this form and choose $\varphi$ such that $\varphi^*(\fra(Y))$ is not complemented in $\fra(X)$. Indeed, choose compact spaces $K$ and $L$ and a continuous surjection $\psi\colon K\to L$ such that $\psi^*(\C(L))$ is not complemented in $\C(K)$. (One can take $K$ to be the Cantor set $\{0,1\}^\en$, $L$ to be the unit interval and $\psi\colon K\to L$ be the standard surjection. Then $\psi^*(C(L))$ is not complemented in $C(K)$, which follows e.g. from \cite[Lemma~2.7]{kal-kub}.) Further, let $\varphi:\M^1(K)\to\M^1(L)$ assign to each $\mu\in \M^1(K)$ its image under $\psi$.
Then $\varphi$ is an affine continuous surjection and $\varphi^*(\fra(\M^1(L)))$ is not complemented in $\fra(\M^1(K))$.
\end{remark}

\section{Proof of the main abstract result}\label{sec:pfmain}

The aim of this section is to prove Theorem~\ref{T:main}. To do that we need several lemmata. We point out that all vector spaces in this section
are supposed to be real. The results of this section can be also used for complex vector spaces if we forget the complex multiplication and look at them as at real spaces.

If $X$ is a compact convex set and $x\in X$, we will denote by $\ep_x$ the respective evaluation functional, i.e.,
$$\ep_x(u)=u(x),\quad u\in\fra(X).$$
Then clearly $\ep_x\in\fra(X)^*$ and $\norm{\ep_x}=1$ for each $x\in X$. Moreover, we have the following lemma.

\begin{lemma}\label{L:dual}
Let $X$ be a compact convex set. Then the following assertions hold.
\begin{itemize}
	\item[(i)] The mapping $\ep\colon x\mapsto \ep_x$ is an affine homeomorphism of $X$ into $(\fra(X)^*,w^*)$. Moreover,
	$\ep(X)=\{\eta\in \fra(X)^*\setsep \eta\ge0 \ \&\ \eta(1)=1\}$.
	\item[(ii)] For any $\eta\in \fra(X)^*$ there are $x_1,x_2\in X$ and $a_1,a_2\ge0$ such that $\eta=a_1\ep_{x_1}-a_2\ep_{x_2}$ and $\norm{\eta}=a_1+a_2$.
	\item[(iii)] The weak topology on $\fra(X)$ coincides with the topology of pointwise convergence on $X$.
\end{itemize}
\end{lemma}

\begin{proof} The assertions (i) and (ii) are proved in \cite[Proposition 4.31(a,b)]{lmns}. The assertion (iii) follows immediately from (ii).

Since we will use it later, we indicate how to find $x_1,x_2\in X$ and $a_1,a_2\ge0$ provided by (ii). Let $\eta\in\fra(X)^*$ be given.
By the Hahn-Banach theorem the functional $\eta$ can be extended to some $\tilde\eta\in\C(X)^*$ with the same norm. By the Riesz representation theorem the functional $\tilde\eta$ is represented by a signed Radon measure $\mu$ on $X$. We set $a_1=\mu^+(X)$ and $a_2=\mu^-(X)$. If  $a_1=0$, let $x_1\in X$ be arbitrary; if $a_1>0$, let $x_1$ be the barycenter of $
\frac{\mu^+}{a_1}$. The point $x_2$ is defined analogously.
\end{proof}

In the sequel we will use the following notation. If $X$ and $Y$ are two convex subsets of some vector spaces, by $\aff(X,Y)$ we denote the set of all affine mappings defined on $X$ with values in $Y$. Further, if $E$ and $F$ are vector spaces, $\lin(E,F)$ will denote the vector space of all linear operators from $E$ to $F$.

\begin{lemma}\label{L:linearizace} Let $X$ be a compact convex set and $F$ a vector space. For any affine mapping $f\colon X\to F$ there is a unique linear map $L_f\colon\fra(X)^*\to F$ such that $L_f(\ep_x)=f(x)$ for each $x\in X$.
Moreover, the operator $L\colon f\mapsto L_f$ is a linear bijection of the space $\aff(X,F)$ onto the space $\lin(\fra(X)^*,F)$.
\end{lemma}

\begin{proof} This essentially follows from \cite[Exercise 4.48]{lmns}, where a similar assertion is shown for scalar functions.
Let us indicate the proof. If $\eta\in \fra(X)^*$, let $x_1,x_2\in X$ and $a_1,a_2\ge0$ be provided by Lemma~\ref{L:dual}(ii).
We have to set
$$L_f(\eta)=a_1f(x_1)-a_2f(x_2).$$
Hence the uniqueness is clear. It remains to observe that this formula correctly defines a linear mapping. To see that the definition of $L_f$ is correct it is enough to check that
\begin{multline*}x_1,x_2,y_1,y_2\in X, a_1,a_2,b_1,b_2\ge 0, a_1\ep_{x_1}-a_2\ep_{x_2}=b_1\ep_{y_1}-b_2\ep_{y_2} \\ \Rightarrow a_1 f(x_1)-a_2f(x_2)=b_1f(y_1)-b_2f(y_2),\end{multline*}
which easily follows from the fact that both $\ep$ and $f$ are affine maps. Now, $L_f$ is clearly affine and $L_f(0)=0$, hence $L_f$ is linear.

It is clear that the operator $L$ is linear. Moreover, if $T:\fra(X)^*\to F$ is a linear mapping, then $T=L_{T\circ\ep}$, hence $L$ is a linear bijection.
\end{proof}

\begin{lemma}\label{L:linspoj}
Let $X$ be a compact convex set and $F$ a topological vector space. Let $L:f\mapsto L_f$ be the operator from Lemma~\ref{L:linearizace}. Then the following assertions hold.
\begin{itemize}
	\item[(i)] $L$ is a homeomorphism of $\aff(X,F)$ onto $\lin(\fra(X)^*,F)$, when both spaces are equipped with the pointwise convergence topology.
	\item[(ii)] $L_f$ is bounded if and only if $f$ is bounded.
	\item[(iii)] $L_f$ is weak$^*$-continuous on $B_{\fra(X)^*}$ if and only if $f$ is continuous.
	\item[(iv)] $L_f\in L_\alpha(\fra(X)^*,F)$ if and only if $f\in\fra_\alpha(X,F)$. Here $L_\alpha(\fra(X)^*,F)=(L(\fra(X)^*,F))_\alpha$, where $L(\fra(X)^*,F)$ is the subspace of $\lin(\fra(X)^*,F)$ consisting of maps which are weak$^*$-continuous on the unit ball.
\end{itemize}
\end{lemma}

\begin{proof} (i) Given $\eta\in \fra(X^*)$ fix $x_1,x_2\in X$ and $a_1,a_2\ge 0$ provided by Lemma~\ref{L:dual}(ii). Then for any  $f\in\aff(X,F)$ we have $L_f(\eta)=a_1f(x_1)-a_2f(x_2)$, thus the mapping $f\mapsto L_f(\eta)$ is continuous in the pointwise convergence topology. It follows that $L$ is continuous.

Conversely, given $x\in X$ we have $L^{-1}(T)(x)=T(\ep_x)$ for $T\in\lin(\fra(X)^*,F)$, hence $T\mapsto L^{-1}(T)(x)$ is continuous in the pointwise convergence topology. It follows that $L^{-1}$ is continuous.

(ii) If $L_f$ is bounded, i.e., if $L_f(B_{\fra(X)^*})$ is bounded in $F$, then $f=L_f\circ \ep$ is also bounded as $\ep(X)\subset B_{\fra(X)^*}$.

Conversely, let $f$ be bounded, i.e., let $f(X)$ be bounded. We will show that $L_f(B_{\fra(X)^*})$ is bounded as well. To this end let $U$ be any neighborhood of zero in $F$. Let $V$ be a balanced neighborhood of zero in $F$ such that $V+V\subset U$. Since $f(X)$ is bounded, there is $\lambda>0$ with $f(X)\subset \lambda V$. Let $\eta\in B_{\fra(X)^*}$ be arbitrary. Fix $x_1,x_2\in X$ and $a_1,a_2\ge 0$ provided by Lemma~\ref{L:dual}(ii). Since $a_1+a_2\le 1$, we have
$a_1 f(x_1)\in \lambda V$ and $-a_2 f(x_2)\in \lambda V$. Then $L_f(\eta)=a_1f(x_1)-a_2f(x_2)\in \lambda V+\lambda V\subset \lambda U$. Hence $L_f(B_{\fra(X)^*})\subset \lambda U$ and the proof is completed.

(iii) The `only if' part is trivial. Let us show the `if' part. Suppose that $f$ is continuous. Set
$$K=X\times X\times\{(a_1,a_2)\in\er^2\setsep a_1\ge0,a_2\ge 0, a_1+a_2\le 1\}.$$
Then $K$ is a compact space. Moreover, define mappings $\psi:K\to B_{\fra(X)^*}$ and  $g:K\to F$ by the formulas
$$\begin{aligned}
\psi(x_1,x_2,a_1,a_2)&=a_1\ep_{x_1}-a_2\ep_{x_2},\\
g(x_1,x_2,a_1,a_2)&=a_1f(x_1)-a_2f(x_2).
\end{aligned}$$
Then $\psi$ is a continuous mapping of $K$ onto $(B_{\fra(X)^*},w^*)$, $g$ is continuous and $g=L_f\circ\psi$.
It follows that $L_f$ is continuous on  $(B_{\fra(X)^*},w^*)$ which was to be proved.

(iv) This assertion follows by transfinite induction from (iii) since, due to (i), $f_n\to f$ pointwise on $K$ if and only if $L_{f_n}\to L_f$ pointwise.
\end{proof}

\begin{lemma}\label{L:bodova}
 Let $X$ be a compact convex set and let $L:f\mapsto L_f$ be the operator provided by Lemma~\ref{L:linearizace} in case $F=\er$.
 Then $L$ is a linear isometry of $\fra_b(X)$, the space of all bounded affine functions on $X$ equipped with the supremum norm, onto $\fra(X)^{**}$. Moreover, it is a homeomorphism from the pointwise convergence topology to the weak$^*$ topology and its restriction to $\fra(X)$ is the canonical embedding of $\fra(X)$ into its second dual.
\end{lemma}

\begin{proof} This is an easy consequence of Lemma~\ref{L:linspoj}.
(It also follows from \cite[Proposition 4.32]{lmns}.)
\end{proof}

In the sequel we will identify $\fra(X)^{**}$ and $\fra_b(X)$.

\begin{proof}[Proof of Theorem~\ref{T:main}.]
Fix a bounded linear projection $P_0:\fra(X)\to \varphi^*(\fra(Y))$ and set $P=(\varphi^*)^{-1}\circ P_0$. Then $P:\fra(X)\to\fra(Y)$ is a bounded linear operator such that $P\circ\varphi^*$ is the identity map on $\fra(Y)$. Recall that $\varphi^*$ is defined by
$\varphi^*(f)=f\circ \varphi$, $f\in\fra(Y)$.

The dual operator $\varphi^{**}\colon\fra(X)^*\to\fra(Y)^*$ satisfies
\begin{equation}
\label{eq:phi**}
\varphi^{**}(\ep_x)=\ep_{\varphi(x)},\quad x\in X.
\end{equation}
Indeed, for any $x\in X$ and $g\in\fra(Y)$ we have
$$\varphi^{**}(\ep_x)(g)=\ep_x(\varphi^*(g))=\ep_x(g\circ \varphi)=g(\varphi(x))=\ep_{\varphi(x)}(g).$$

Let $F$ be a topological vector space and let $f:X\to F$ be affine.
We define an affine map $Qf:Y\to F$ by the formula
\begin{equation}\label{eq:Q}
Qf(y)=L_f(P^*\ep_y), \quad y\in Y.
\end{equation}
It is clear that $Qf$ is affine since it is a composition of three affine maps. Moreover, $Qf\in\fra_\alpha(Y,F)$ whenever
$f\in\fra_\alpha(X,F)$. Indeed, the mapping $\ep$ is continuous from $Y$ to $(\fra(Y)^*,w^*)$, $P^*$ is weak$^*$-to-weak$^*$ continuous and $P^*(\ep(Y))\subset \norm{P} B_{\fra(X)^*}$. Further, if $f\in\fra_\alpha(X,F)$, then by Lemma~\ref{L:linspoj}(iv), $L_f\in L_\alpha(\fra(X)^*,F)$, hence, in particular, $L_f\in\fra_\alpha((\norm{P}B_{\fra(X)^*},w^*),F)$. Thus we get in this case $Qf\in\fra_\alpha(Y,F)$.

The proof will be completed if we show that $Q(g\circ\varphi)=g$ for each affine $g:Y\to F$. So, fix an affine map $g:Y\to F$ and $y\in Y$. Then $P^*(\ep_y)\in\fra(X)^*$, hence, by Lemma~\ref{L:dual}, we can find $x_1,x_2\in X$ and $a_1,a_2\ge 0$ such that $P^*(\ep_y)=a_1\ep_{x_1}-a_2\ep_{x_2}$ and $a_1+a_2\le\norm{P}$. We have
$$\begin{aligned}
Q(g\circ\varphi)(y)&=L_{g\circ\varphi}(P^*\ep_y)=a_1(g\circ\varphi)(x_1)-a_2(g\circ\varphi)(x_2)\\
&=a_1g(\varphi(x_1))-a_2g(\varphi(x_2))=a_1 L_g(\ep_{\varphi(x_1)})-a_2L_g(\ep_{\varphi(x_2)})\\
&=L_g(a_1\ep_{\varphi(x_1)}-a_2\ep_{\varphi(x_2)})=L_g(a_1 \varphi^{**}(\ep_{x_1})-a_2\varphi^{**}(\ep_{x_2}))\\
&=L_g(\varphi^{**}(a_1\ep_{x_1}-a_2\ep_{x_2}))=L_g(\varphi^{**}(P^*(\ep_y)))\\
&=L_g( (P\circ\varphi^*)^*(\ep_y))=L_g(\ep_y)=g(y).
\end{aligned}$$
Indeed, the first four equalities follow from the definitions. The fifth one follows from the linearity of $L_g$.
The sixth equality is a consequence of \eqref{eq:phi**}, the seventh one follows from the linearity of $\varphi^{**}$.
The next one follows from the choice of $a_1,a_2,x_1,x_2$. In the following two equalities we use that $\varphi^{**}\circ P^*=(P\circ\varphi^*)^*=(\id_{\fra(Y)})^*=\id_{\fra(Y)^*}$. Finally, the last one follows from the definition of $L_g$.
\end{proof}

\section{Applications to affine classes of strongly affine Baire mappings}

The aim of this section is to prove Theorem~\ref{T:appl}. We will need several lemmata.

The first one shows an easy correspondence between complementability of Banach spaces and complementability of spaces of affine continuous functions.

\begin{lemma}\label{L:complem}
Let $E_2$ be a complemented subspace of a Banach space $E_1$. Let $\pi:E_1^*\to E_2^*$ be the restriction mapping.
Then $\pi^*(\fra(B_{E_2^*},w^*))$ is complemented in $\fra(B_{E_1^*},w^*)$.
\end{lemma}

\begin{proof} Let us first give the proof for real spaces. Let $P\colon E_1\to E_2$ be a bounded linear projection.
For any $f\in\fra(B_{E_1^*},w^*)$ there is a unique $x_f\in E_1$ such that $f(x^*)=f(0)+x^*(x_f)$ for any $x^*\in B_{E_1^*}$.
Hence we can define a bounded linear operator $Q:\fra(B_{E_1^*},w^*)\to\fra(B_{E_2^*},w^*)$ by
\[
Qf(x^*)=f(0)+x^*(Px_f),\quad x^*\in B_{E_2^*}, f\in \fra(B_{E_1^*}, w^*).
\]
Then $\pi^*\circ Q$ is the required projection.

Now suppose that $E_1$ and $E_2$ are complex spaces. Again, let $P\colon E_1\to E_2$ be a bounded linear projection.
If $f\in\fra((B_{E_1^*},w^*),\ce)$, then there are uniquely determined $x_f,y_f\in E_1$ such that
$$f(x^*)=f(0) + x^*(x_f)+\ov{x^*(y_f)},\quad x^*\in B_{E_1^*},$$
and, moreover, the assignment $f\mapsto x_f$ is a bounded linear map and $f\mapsto y_f$ is a bounded conjugate-linear map (this follows, for example,
from \cite[Lemma 3.11(a)]{l1pred}).
 Hence, we can define a bounded linear operator
 \[
 Q_0\colon\fra((B_{E_1^*},w^*),\ce)\to\fra((B_{E_2^*},w^*),\ce)
 \]
 by
\[
Q_0f(x^*)=f(0)+x^*(Px_f)+\ov{x^*(Py_f)},\quad x^*\in B_{E_2^*}, f\in \fra((B_{E_1^*},w^*),\ce).
\]
Then a required bounded linear projection $Q:\fra(B_{E_1^*},w^*)\to \pi^*(\fra(B_{E_2^*},w^*))$ can be defined by $Qf=\pi^*(\Re Q_0f)$.
\end{proof}

The second lemma is a trivial observation on isomorphic spaces.

\begin{lemma}\label{L:isom} Let $T\colon E_1\to E_2$ be a surjective isomorphism of Banach spaces and $F$ be a vector space. Then there is a unique linear map $\tilde T\colon \aff(B_{E_1^*},F)\to \aff(B_{E_2^*},F)$ such that for any real-linear map $L:E_1^*\to F$ we have $\tilde T(L\r_{B_{E_1^*}})=(L\circ T^*)\r_{B_{E_2^*}}$.

Moreover, $\tilde T$ is a linear bijection and, if $F$ is a topological vector space and $f\in\aff(B_{E_1^*},F)$, the following assertions hold.
\begin{itemize}
	\item $\tilde T$ is a homeomorphism when both spaces are equipped with the pointwise convergence topology.
	
	\item $\tilde T f$ is continuous if and only if $f$ is continuous.
	\item $\tilde T f\in\C_\alpha((B_{E_2^*},w^*),F)$ if and only if $f\in\C_\alpha((B_{E_1^*},w^*),F)$.
	\item $\tilde T f\in\fra_\alpha((B_{E_2^*},w^*),F)$ if and only if $f\in\fra_\alpha((B_{E_1^*},w^*),F)$.
\end{itemize}

Further, if $F$ is \fr space,
\begin{itemize}
\item $\tilde T f$ is strongly affine if and only if $f$ is strongly affine.
\end{itemize}
\end{lemma}

\begin{proof} Let $f\in\aff(B_{E_1^*},F)$. Then there is a real-linear mapping $U_f\colon E_1^*\to F$  such that $f=f(0)+U_f\r_{B_{E_1^*}}$. It is enough to set $\tilde T f=f(0)+(U_f\circ T^*)\r_{B_{E_2^*}}$.
\end{proof}

We will need also the following lemma which summarizes complex versions of some lemmata from the previous section.

\begin{lemma}\label{L:dualcomplex} Let $X$ be a compact convex set.
\begin{itemize}
	\item[(i)] For any $\eta\in\fra(X,\ce)^*$ there exist $x_1,x_2,x_3,x_4\in X$ and $a_1,a_2,a_3,a_4\ge 0$ such that $a_1+a_2+a_3+a_4\le2\norm{\eta}$ and $\eta=a_1\ep_{x_1}-a_2\ep_{x_2} + i (a_3\ep_{x_3}-a_4\ep_{x_4})$.
	\item[(ii)] If $F$ is a complex vector space and $f\in\aff(X,F)$, then there is a unique $L_f\in \lin(\fra(X,\ce)^*,F)$ such that $L_f(\ep_x)=f(x)$ for each $x\in X$. Moreover, the operator $L\colon f\mapsto L_f$ is a linear bijection.
	\item[(iii)] If $F$ is a complex topological vector space, then the operator $L$ has the properties from Lemma~\ref{L:linspoj}.
\end{itemize}
\end{lemma}

\begin{proof} (i) This assertion can be derived from Lemma~\ref{L:dual}(ii). However, we present a direct proof, since we will later use the construction. Let $\eta\in\fra(X,\ce)^*$. By the Hahn-Banach theorem we can choose $\tilde\eta\in\C(X,\ce)^*$ extending $\eta$ and having the same norm. Fix a complex Radon measure $\mu$ on $X$ which represents $\tilde\eta$ (by the Riesz representation theorem). We set
$$a_1=(\Re\mu)^+(X),\ a_2=(\Re\mu)^-(X),\ a_3=(\Im\mu)^+(X),\ a_4=(\Im\mu)^-(X).$$
If $a_1=0$, let $x_1\in X$ be arbitrary. If $a_1>0$, let $x_1$ be the barycenter of $(\Re\mu^+)/a_1$. Similarly we define $x_2$, $x_3$ and $x_4$.

(ii) Let $f\in\aff(X,F)$. Given $\eta\in\fra(X,\ce)^*$, fix a representation of $\eta$ by (i). We have to set
\[
L_f(\eta)=a_1f(x_1)-a_2f(x_2)+i(a_3f(x_3)-a_4f(x_4)).
\]
Therefore, the uniqueness is clear. To see that $L_f$ is well defined, it is enough to check that
\[
\begin{array}{c}
x_1,x_2,x_3,x_4,y_1,y_2,y_3,y_4\in X, a_1,a_2,a_3,a_4,b_1,b_2,b_3,b_4\ge 0,\\
 a_1\ep_{x_1}-a_2\ep_{x_2}+i(a_3\ep_{x_3}-a_4\ep_{x_4})=b_1\ep_{y_1}-b_2\ep_{y_2}+i(b_3\ep_{y_3}-b_4\ep_{y_4}) \\
\Downarrow \\
a_1 f(x_1)-a_2f(x_2)+i(a_3 f(x_3)-a_4f(x_4))\\
\hskip3cm=b_1f(y_1)-b_2f(y_2)+i(b_3f(y_3)-b_4f(y_4)),
\end{array}
\]
which easily follows from the fact that $\ep$ and $f$ are affine. It is clear that $L_f$ is affine. Since $L_f(0)=0$ and $L_f(i\eta)=i L_f(\eta)$ for $\eta\in\fra(X,\ce)^*$, we get that $L_f$ is linear. Further, the operator $L$ is obviously linear.
Since for any $T\in\lin(\fra(X,\ce)^*,F)$ we have $T=L_{T\circ\ep}$, $L$ is a linear bijection.

(iii) The proof is analogous to the real case.
\end{proof}

\begin{lemma}\label{L:sa} Let $X$ be a compact convex set, $F$ a \fr space over $\ef$ and $f\colon X\to F$ a  strongly  affine mapping which belongs to $\C_\alpha(X,F)$ for some $\alpha<\omega_1$. Let $L_f\colon \fra(X,\ef)^*\to F$ be the linear map provided by Lemma~\ref{L:dual} or Lemma~\ref{L:dualcomplex}. Then $L_f\r_{(B_{\fra(X,\ef)^*},w^*)}$ is a strongly affine mapping
which belongs to $\C_\alpha((B_{\fra(X,\ef)^*},w^*),F)$.
\end{lemma}

\begin{proof} Let us denote by $\M(X,\ef)$ the space of all $\ef$-valued Radon measures on $X$ considered as the dual space to $\C(X,\ef)$ equipped with the weak$^*$ topology. For any bounded Baire function $g\colon X\to F$ let us define a mapping $\tilde g:B_{\M(X,\ef)}\to F$ by setting $\tilde g(\mu)=\int_X g\di\mu$. The integral is considered in the Pettis sense, see \cite[Section 1.3]{l1pred}. The mapping $\tilde g$ is well defined by \cite[Lemma 3.6]{l1pred}.

Moreover, $\tilde g\in\fra_\alpha(B_{\M(X,\ef)},F)$ whenever $g\in\C_\alpha(X,F)$. We will prove it by transfinite induction on $\alpha$.
Let $\alpha=0$, i.e., let $g$ be continuous. Then for each $\tau\in F^*$ the function $\tau\circ g$ is continuous, thus
the mapping
$$\mu\mapsto \tau(\tilde g(\mu))=\int_X \tau\circ g\di\mu,\quad \mu\in B_{\M(X,\ef)},$$
is continuous by the very definition of the weak$^*$ topology. Thus $\tilde g$ is continuous from $B_{\M(X,\ef)}$ to the weak topology of $F$. Further, the range of $\tilde g$ is contained in the closed absolutely convex hull of $g(X)$ (by \cite[Lemma 3.5(c)]{l1pred}). Since $g(X)$ is compact, we deduce that $\tilde g(X)$ is relatively compact  by \cite[Proposition 6.7.2]{jarchow}. Hence the weak topology coincides with the original one on $\tilde g(X)$, so $\tilde g$ is continuous.
Finally, it is clear that $\tilde g$ is affine.

Let $\alpha\in(0,\omega_1)$ be such that the statement hold for each $\beta<\omega_1$. Let $g\in\C_\alpha(X,F)$ be a bounded function. Then there is a uniformly bounded sequence $(g_n)$ in $\bigcup_{\beta<\alpha} \C_\beta(X,F)$ pointwise converging to $g$. Indeed, by \cite[Lemma 3.1(c)]{l1pred} we get that $g\in\C_\alpha(X,\co g(X))$ and the set $\co g(X)$ is bounded.
Now it follows from \cite[Theorem 3.7]{l1pred} that $\tilde{g_n}\to \tilde g$ pointwise. Since by the induction hypothesis we have
$\tilde{g_n}\in\bigcup_{\beta<\alpha} \fra_\beta(B_{\M(X,\ef)},F)$, we conclude $\tilde g\in \fra_\alpha(B_{\M(X,\ef)},F)$.

We return to the proof of the lemma. Let $X,F,f,\alpha$ be as in the statement. Since $f$ is bounded by \cite[Lemma 4.1]{l1pred}, we get $\tilde f\in\fra_\alpha(B_{\M(X,\ef)},F)$. Further, let $\pi:\M(X,\ef)\to\fra(X,\ef)^*$ be the restriction map (recall that $\M(X,\ef)$ is identified with $\C(X,\ef)^*$). We claim that
\begin{equation}\label{eq:sloz}
\tilde f=(L_f \circ \pi)\r_{B_{\M(X,\ef)}}.
\end{equation}
Since the mappings on both sides are restrictions of linear operators, it is enough to check that they agree on probability measures. Let $\mu\in\M^1(X)$. Then $\pi(\mu)=\ep_{r(\mu)}$, hence
$$L_f(\pi(\mu))=L_f(\ep_{r(\mu)})=f(r(\mu))=\int f\di\mu=\tilde f(\mu),$$
where we used the definitions and the assumption that $f$ is strongly affine.

Finally, $\pi$ is a continuous affine surjection. Since $\tilde f\in\fra_\alpha(B_{\M(X,\ef)},F)$, by Theorem B we conclude that $L_f\r_{(B_{\fra(X,\ef)^*},w^*)}\in \C_\alpha((B_{\fra(X,\ef)^*},w^*),F)$ and  $L_f\r_{(B_{\fra(X,\ef)^*},w^*)}$ is strongly affine by \cite[Proposition 5.29]{lmns} and \cite[Fact 1.2]{l1pred}.
\end{proof}

\begin{proof}[Proof of Theorem~\ref{T:appl}.]
(a) Let $E_1$ be an $L_1$-predual, $E$ a Banach space isomorphic to a complemented subspace of $E_1$ and $X=(B_{E^*},w^*)$.
Further, let $F$ be a \fr space, $\alpha<\omega_1$ and $f\in\C_\alpha(X,F)$ a strongly affine function. Fix $T:E\to E_1$ an isomorphism of $E$ onto a complemented subspace of $E_1$. Denote $E_2=T(E)$ and let $\pi:E_1^*\to E_2^*$ denote the restriction mapping.

Let $g=\tilde T f$ (see Lemma~\ref{L:isom}). Then $g$ is strongly affine and $g\in\C_\alpha((B_{E_2^*},w^*),F)$ (by Lemma~\ref{L:isom}). Hence $g\circ\pi$ is a strongly affine function in  $\C_\alpha((B_{E_1^*},w^*),F)$.
By \cite[Theorem 2.5]{l1pred} we get $g\circ\pi\in\fra_{1+\alpha}((B_{E_1^*},w^*),F)$. Therefore, by Lemma~\ref{L:complem} and Theorem~\ref{T:main} we get $g\in\fra_{1+\alpha}((B_{E_2^*},w^*),F)$. By Lemma~\ref{L:isom} we conclude that $f\in\fra_{1+\alpha}(X,F)$.

If $\ext B_{E_1^*}$ is moreover a weak$^*$ $F_\sigma$-set, from \cite[Theorem 2.5]{l1pred} we obtain
$g\circ\pi\in\fra_{\alpha}((B_{E_1^*},w^*),F)$,
therefore we conclude $f\in\fra_{\alpha}(X,F)$.

(b) Let us start by the real case. Hence, let $E_1$ be a real $L_1$-predual, $X$ a compact convex space, $T\colon\fra(X)\to E_1$
an isomorphism onto a complemented subspace. Let $F$ be a \fr space and $f\colon X\to F$ a strongly affine function which belongs to $\C_\alpha(X,F)$. Without loss of generality we can suppose that $F$ is real (if $F$ is complex, we can forget
complex multiplication and consider $F$ as a real space). By Lemma~\ref{L:sa} the mapping $L_f\r_{B_{\fra(X,\ef)^*}}$ is strongly affine and belongs to
$\C_\alpha((B_{\fra(X,\ef)^*},w^*),F)$. Therefore by the already proved case (a) we get $L_f\r_{B_{\fra(X,\ef)^*}}\in\fra_{1+\alpha}((B_{\fra(X,\ef)^*},w^*),F)$, hence $f\in\fra_{1+\alpha}(X,F)$.

If $\ef=\ce$ and $F$ is complex, the proof can be done in the same way as in the real case. 	

Finally, suppose that $\ef=\ce$ and $F$ is real. Let $F_{\ce}$ be the complexification of $F$. By the previous case
we get $f\in\fra_{1+\alpha}(X,F_{\ce})$. Since the canonical projection of $F_{\ce}$ onto $F$ is real-linear and continuous,
we get $f\in\fra_{1+\alpha}(X,F)$.

If $\ext B_{E_1^*}$ is weak$^*$ $F_\sigma$, $1+\alpha$ can be replaced by $\alpha$ as  in the case (a).
\end{proof}


\end{document}